\documentclass[11pt]{amsart}
\usepackage{amssymb,amsmath}
\usepackage[nobysame]{amsrefs}
\def\Dt{\partial_t}

\def\eb{\varepsilon}

\def\R {\mathbb{R}}

\def \pt {\partial_t}

\def\<{\left<}
\def\>{\right>}

\def\({\left(}
\def\){\right)}

\newtheorem{proposition}{Proposition}[section]
\newtheorem{theorem}[proposition]{Theorem}

\newtheorem{lemma}[proposition]{Lemma}

\theoremstyle{definition}
\newtheorem{definition}[proposition]{Definition}

\newtheorem{remark}[proposition]{Remark}

\numberwithin{equation}{section}

\def \no#1#2#3 {{\bf #1} (#3), #2.}
\def \eds#1#2#3 {#1, #2, #3.}

\title[Inertial manifold for the modified-Leray-$\alpha$ model on a 3D torus] {Inertial manifolds for the 3D modified-Leray-$\alpha$ model with periodic boundary conditions }
\author[A.Kostianko]{ Anna Kostianko${}^1$}
\address{${}^1$
University of Surrey, Department of Mathematics,
Guildford, GU2 7XH, United Kingdom.}
\email{a.kostianko@surrey.ac.uk}
\begin{document}
\begin{abstract} The existence of an inertial manifold for the modified Leray-$\alpha$ model with periodic boundary conditions in three-dimensional space is proved by using the so-called spatial averaging principle. Moreover, an adaptation of the Perron method for constructing  inertial manifolds  in the particular case of zero spatial averaging is suggested.
\end{abstract}
\maketitle
\tableofcontents
\section{Introduction}\label{s0}

It is well known that the Navier-Stokes equation (NSE) is the main equation of the dynamics of the turbulent flow, but for the current state of art we are unable to solve NSE analytically and
even the global well-posedness of this equation in 3D is an open problem which is indicated by the Clay Mathematics Institute as one of 7 Millennium problems in mathematics. One of the main difficulties arising here is a strong impact of the higher modes to the leading order dynamics through the non-linearity which a priori may destroy the regularity of a solution and lead to the formation of singularities.
\par
 To overcome this difficulty, a number of modified/averaged systems has been proposed  in order to capture the leading dynamics of the flow on the one hand and somehow suppress the higher modes on the other hand. One of such systems is the so-called modified Leray-$\alpha$ (ML-$\alpha$) model:
\begin{equation}\label{main_eq}
 \begin{cases}
  u_t - \nu \Delta u + (u \cdot \nabla)\bar{u} + \nabla p=f,\\
 \nabla \cdot u = \nabla \cdot \bar{u}  = 0,\\
  u= \bar{u}- \alpha^2 \Delta \bar{u},\\
  u(0)=u_0,
 \end{cases}
\end{equation}
where the unknowns are: the fluid velocity vector $u$, the "filtered" velocity vector $\bar{u}$ and "filtered" pressure scalar $p$; given parameters are: $\nu >0$, which is the constant kinematic viscosity, and $\alpha >0$, which is a length scale parameter responsible for the width of the filter; the vector field $f$ is a given body forcing, which is assumed to be time independent, and $u_0=u_0(x)$ is the given initial velocity.

This system was introduced in \cite{Il_Ti} and it was  
 shown  there that consideration of  the ML-$\alpha$ model as
a closure model to Reynolds averaged equations in turbulent channels and pipes leads to the similar reduced system as   the Leray-$\alpha$ (see, e.g., \cite{ChHOT}) and the Navier-Stokes-$\alpha$ models (see, e.g., \cites{ChFHOTW, FHT}). Thus using \eqref{main_eq} is equally effective as using other $\alpha$ subgrid scale models in infinite channels and pipes. Moreover some important analytical properties of the ML-$\alpha$ are  proved in \cite{Il_Ti} , in particular interest for us is a global well-posedness in both 2D and 3D and the  existence of a global attractor in the properly chosen phase space.
\par
In the present study we are interested in the long-time behaviour of solutions of the system \eqref{main_eq} in \emph{three-dimensional case}. 
More precisely we are going to show the existence of an inertial manifold (IM) for the ML-$\alpha$ system subject to periodic boundary conditions.
 The analogous result for  system \eqref{main_eq} in the 2D case  was proved in \cite{HGT}.
 \par
Recall that the large class of dissipative systems possesses a resemblance of a finite-dimensional system. The notion of an inertial manifold was introduced exactly to formulate in a rigorous mathematical way what it means a finite-dimensionality for the infinite-dimensional system. By definition, IM is an invariant, finite-dimensional  Lipschitz continuous manifold, that possesses exponential tracking property, the precise definition is given in Section 2. Therefore, if such object exists we may restrict our system to the inertial manifold, and since IM is finite-dimensional, we obtain system of ODEs (so-called inertial form of the given evolution equation) that describes the limit dynamics. Since we are interested only in a long-time behaviour of the solutions, namely on the global attractor, we may modify our system outside of the absorbing ball and construct the IM for the new system, nevertheless it will not correspond the original system in general but will capture all limit dynamics which is our goal.

The concept of inertial manifolds was proposed by C. Foias, G. R. Sell and R. Temam in \cite{FST} and after that applied and evolved in many works, see, e.g., \cites{CFNT1, CFNT2, EKZ, FNS, FSTi, HGT, Mikl, Rom, Zel1}.  These papers are based on the so-called  {\it spectral gap} assumption which is a sufficient condition for the existence of an IM.  However, for more or less general underlying domains it is usually fulfilled only for parabolic equations in one-dimensional case. In the case of equation \eqref{main_eq}, the spectral gap condition is valid for a 2D torus (which is used for the proving of the existence of an IM in \cite{HGT}), but it already fails for the case of a 3D torus. Thus, the situation in the 3D case is essentially different (in comparison with the 2D case considered in \cite{HGT}), so the result of \cite{HGT} cannot be extended in a straightforward way to the 3D case  and new ideas and methods are required to handle it.
\par
  In the present paper we develop further the spatial averaging method which was originally proposed by J. Mallet-Paret and G. R. Sell in \cite{mal-par} for the scalar reaction-diffusion equation on a 3D torus and after that significantly simplified in \cite{Zel}. Roughly speaking, this method is based on the observation that only the action of the derivative $F'(u)$ of the non-linearity on the intermediate "modes" which are near the spectral gap threshold is really important for the classical construction of the IM, so the assumptions on the action on the remaining modes can be essentially relaxed. In particular, in the case where this intermediate action is close in a proper sense to a scalar operator $\<F'(u)\>$ (which is called a spatial average of the operator $F'(u)$), the spectral gap assumption can be removed, see \cite{mal-par,Zel} for more details.   
   Previously this method has been applied only to scalar equations, see \cites{Zel-Ko, kwean, mal-par, Zel}, because  in the  case of systems it is in general impossible to find a scalar operator $\<F'(u)\>$ to fulfil the condition mentioned above (the scalar operator $\<F'(u)\>$ naturally becomes a {\it matrix} in this case which is not allowed in the spatial averaging method). 
   \par
   Fortunately, we succeed to overcome this problem for equations \eqref{main_eq} due to the specific form of the Navier-Stokes non-linearity which allows us to take $\<F'(u)\>=0$. One more non-trivial problem arising here is the choice of the proper cut-off of the non-linearity $F$. We recall that the abstract constructions of the IM deal only with the globally Lipschitz non-linearities, so the initial non-linearity should be somehow cut off outside of the global attractor or absorbing ball in order to fulfil this condition. Being relatively straightforward for the case where the spectral gap condition is satisfied, this cut of procedure becomes delicate and non-trivial when the spatial averaging method is used since we should not destroy the possibility to approximate $F'(u)$ by scalar operators (on the intermediate modes) under the cut off procedure, see \cite{mal-par} for more details. Actually, the original cut-off procedure suggested in \cite{mal-par} seems not working in our case and to overcome this problem, we suggest below a new general method which is based on a direct truncation of the Fourier modes in the non-linearity $F(u)$ which does not destroy the condition $\<F'(u)\> = 0$ in the case of equations \eqref{main_eq}. We note that this method is applicable to the scalar cases mentioned before and leads to essential simplifications there as well. 
   \par 
The paper is organized as follows: in section 2 we formulate some necessary background related with the Navier-Stokes equations and give basic definitions. In section 3, some dissipativity and regularity results for equations \eqref{main_eq} in higher energy spaces (including the existences and smoothness of the corresponding global attractor) which
are necessary for our method of constructing the IM are proved. The alternative approach of constructing inertial manifolds via spatial averaging which does not utilize the cone property and the graph transform, but based on the Perron method, is developed in Section 4. 
 Finally, Section 5 is devoted to the proof of our main result on the existence of an IM for ML-$\alpha$ model \eqref{main_eq} (in particular, our new cut-off procedure is presented here) and to discussion of the analogous results concerning some other $\alpha$ models.


\section{Preliminaries}
In this section, we recall some basic definitions and standard results which will be used thought out the paper.
Let $\mathbb{T}^3=[-\pi,\pi]^3$ and let
   $H^s(\mathbb T^3)$, $s\in\R$,  be the classical Sobolev spaces on a torus $\mathbb T^3$.
Then every function $u \in (L^2(\mathbb T^3))^3$ can be split into the Fourier series
\begin{equation}
u(x)= \sum_{n\in \mathbb{Z}^3} u_n e^{in\cdot x}, \ \ \ u_n = \frac{1}{(2 \pi)^3}\int_{\mathbb{T}^3}u(x) e^{in \cdot x}\in \mathbb C^3,
\end{equation}
where $n \cdot x = \sum_{j=1}^3 n_j \cdot x_j$ is a standard inner product in $\mathbb{R}^3$. Moreover, due to the Plancherel theorem, the norm in the space $(H^s(\mathbb T^3))^3$ is defined by
\begin{equation}
\|u\|_{(H^s(\mathbb T^3))^3}^2=|u_0|^2+\sum_{n\in\mathbb Z^3\, n\ne0}|n|^{2s}|u_n|^2,\ \ |n|^2:=n_1^2+n_2^2+n_3^2.
\end{equation}
The Leray-Helmholtz orthoprojector $P: (L^2(\mathbb T^3))^3\to H:=P(L^2(\mathbb T^3))^3$ to divergent free vector fields with zero mean can be defined as follows
\begin{equation}
P u:=\sum_{n\in\mathbb Z^3\, n\ne0}P_n u_n e^{ix\cdot n},\ \ u=\sum_{n\in\mathbb Z^3}u_ne^{in\cdot x}
\end{equation}
and the $3\times3$-matrices $P_n$ are defined by
\begin{equation}\label{0.pmatrix}
P_n:=\frac1{|n|^2}\(\begin{matrix} n_2^2+n_3^2&,&-n_1n_2&,&-n_1n_3\\-n_1n_2&,&n_1^2+n_3^2&,&-n_2n_3
\\-n_1n_3&,&-n_2n_3&,&n_1^2+n_2^2\end{matrix}\).
\end{equation}
All of the matrices $P_n$ are orthonormal projectors in $\mathbb C^3$, so their norms are equal to one: $\|P_n\|=1$, for all $n\ne0$. Moreover, the projector $P$ commutes with the Laplacian:
$$
\Delta P=P\Delta
$$
and we may define the Stokes operator $A:=-P\Delta$ as the restriction of the Laplacian to the divergence free vector fields. The domain $D(A)$ of this operator is given by
$$
D(A):= (H^2(\mathbb{T}^3))^3 \cap \{\nabla \cdot u = 0\}\cap\{\<u\>=0\}.
$$
We define a scale of Hilbert spaces $H^s:=D(A^{s/2})$, $s\in\R$. Then, as not difficult to show
\begin{equation}
Au = -P \Delta u = -\Delta u, \ \ \text{ for all } u \in D(A).
\end{equation}


 $$
 H^s:=D(A^{s/2}) = (H^s(\mathbb{T}^3))^3\cap \{\nabla \cdot u = 0\}\cap\{\<u\>=0\}
 $$
(see, e.g., \cite{Tem}) and, due to the Parseval equality, the norm in this space can be defined by
\begin{equation}\label{Poin}
 \|u\|^2_{H^s} = \sum_{n\in \mathbb{Z}^3, n \not = 0} |n|^{2s}u_n^2,\ \ u\in H^s.
\end{equation}
For $u_1, u_2 \in H^1$ we define the standard bilinear form associated with the Navier-Stokes equation:
\begin{equation}
B(u_1,u_2):=P((u_1 \cdot \nabla)u_2).
\end{equation}
The bilinear form $B: H^1 \times H^1 \to H^{-1}$ is continuous and satisfies the following estimates:
\begin{equation}\label{est_B}
 |(B(u_1,u_2),u_3)|\le c \|u_1\|_{H^1}\|u_2\|_{H^1}^{\frac{1}{2}}\|Au_2\|^{\frac{1}{2}}_H \|u_3\|_H,
\end{equation}
for every $u_1 \in H^1$, $u_2 \in H^2$ and $u_3 \in H$ (here and below $(u,v)$ stands for the standard inner product in $(L^2(\mathbb T^3))^3$),
\begin{equation}\label{est_B_2}
 |(B(u_1,u_2),u_3)| \le c\|u_1\|^{1/2}_H\|u_1\|^{1/2}_{H^1}\|u_2\|_{H^1} \|u_3\|_{H^1},
\end{equation}
and
\begin{equation}
(B(u_1,u_2), u_3) = - (B(u_1,u_3), u_2), \ \text{ for every } u_1, u_2, u_3 \in H^1
\end{equation}
and, consequently,
\begin{equation}\label{est_B_3}
(B(u_1,u_2), u_2) = 0, \ \text{ for every } u_1, u_2 \in H^1.
\end{equation}
Applying the Helmholtz-Leray orthogonal projection $P$ to equation \eqref{main_eq} and observing that $u = \bar{u} + \alpha^2 A \bar{u}$ we get
\begin{equation}\label{main_mod}
 \begin{cases}
   u_t + \nu A u + B(u,\bar{u}) = f,\\
   u= \bar{u}+ \alpha^2 A \bar{u}.\\
   u(0)=u_0,
 \end{cases}
\end{equation}
where we assume that $P f = f$. We always can do so due to the modification of the pressure $p$ in such a way that it includes the gradient part of $f$.

Also, from the second equation of \eqref{main_mod}, we have
\begin{equation} \label{bar_u_u}
\frac1{1+\alpha^2}\|u\|^2_H\le\|\bar u\|_{H^2} \le \frac1{\alpha^2} \| u\|_{H}.
\end{equation}
Let us define the main object of this paper, namely, the inertial manifold (IM) associated with the modified-Leray-$\alpha$ model.
\begin{definition}\label{defman}
A subset $\mathcal{M} \subset H$ is called an inertial manifold for problem \eqref{main_mod} if the following conditions are satisfied:

1. $\mathcal{M}$ is invariant with respect to the solution semigroup $S(t)$, i.e. $S(t)\mathcal{M}\subset \mathcal{M}$, for all $t\ge 0$;

2. $\mathcal{M}$ is a finite-dimensional Lipschitz manifold,
i.e. there exist an open set $V\subset\mathcal{P}_NH$ and  a Lipschitz continuous function $\Phi: \mathcal{P}_NH \to \mathcal{Q}_N H$ such that
\begin{equation}
\mathcal{M}:=\{u_+ + \Phi(u_+),\, u_+ \in V\},
\end{equation}
where $\mathcal{P}_N$ is the orthoprojector to the first $N$ Fourier modes and $\mathcal{Q}_N = Id -\mathcal{P}_N$.

3. $\mathcal{M}$ attracts exponentially all solutions of \eqref{main_mod}, i.e. there exist positive constants $C$ and $\gamma$ such that for every $u_0 \in H$ there exists $v_0 \in \mathcal{M}$ such that
\begin{equation}
\|S(t)u_0-S(t)v_0\|_H \le C e^{-\gamma t}\|u_0-v_0\|_H,\  t\ge0.
\end{equation}
\end{definition}

\section{Dissipative estimates}
This section is devoted to the proof of the boundedness of the attractor for the semigroup $S(t)$ associated with equation \eqref{main_mod} in $H^3$-norm. These estimates are important in the  proving of boundedness and Lipschitz continuity of the cut-off version of the non-linearity, which is one of the conditions for the existence of an IM (see Theorem \eqref{main_th}).
The reasoning provided here is formal and can be justified by using, e. g., Galerkin approximation scheme.
In \cite{Il_Ti} there were obtained $H^1$ and $H^2$-estimates on $\bar{u}(t)$  which we are going to use in order to prove the boundedness of the $H^3$-norm of $w:=u-v$, where $u$ is a solution of the  problem \eqref{main_mod} and $v$ is a solution of the stationary problem
\begin{equation}\label{stat_pr}
 \begin{cases}
    \nu A v + B(v,\bar{v}) =f,\\
    v =\bar{v} + \alpha^2 A \bar{v}.
 \end{cases}.
\end{equation}
We start with proving the required regularity result for the auxiliary equation \eqref{stat_pr}.
\begin{lemma} Let the function $f\in H$ and $\alpha,\nu>0$. Then, there exists at least one weak solution $v\in H^1$ of the elliptic problem \eqref{stat_pr} and any  solution $v\in H^2$ and the following estimate holds
\begin{equation}\label{3.disV}
\|v\|_{H^2}\le C\|f\|_H(1+\|f\|_H^{4/3}),
\end{equation}
where the constant $C$ depends on $\alpha$ and $\nu$ and is independent of $v$ and $f$.
\end{lemma}
\begin{proof} We give below only the formal derivation of estimate \eqref{3.disV} which can be justified in a standard way. First, we multiply \eqref{stat_pr} by $\bar v$ and integrate over $\mathbb T^3$. Then, using that $v=\bar v+\alpha^2A\bar v$ together with \eqref{est_B_3}, we get
$$
\nu\|\bar v\|^2_{H^1}+\nu\alpha^2\|\bar v\|^2_{H^2}=(f,\bar v)
$$
and, consequently, by the Cauchy-Schwarz inequality and inequality \eqref{bar_u_u},
\begin{equation}\label{3.l2}
\|v\|^2_H\le(1+\alpha^2)\|\bar v\|^2_{H^2}\le \frac{1+\alpha^2}{4\alpha^2\nu^2}\|f\|^2_H.
\end{equation}
In order to obtain the $H^2$-estimate for $v$, we multiply equation \eqref{stat_pr} by $Av$ and estimate the non-linear term as follows:
\begin{multline}
|(B(v,\bar v),Av)|\le C\|v\|_H\|Av\|_H\|\nabla\bar v\|_{L^\infty}\le\\\le C_1\|v\|_H\|\bar v\|_{H^2}^{3/4}\|\bar v\|_{H^4}^{1/4}\|Av\|_{H}\le C_2\|v\|_H^{7/4}\|Av\|^{5/4}_H,
\end{multline}
where we have used the interpolation inequality $\|u\|_{L^\infty}\le C\|u\|_{H^1}^{3/4}\|u\|_{H^3}^{1/4}$ and inequality \eqref{bar_u_u}. This together with the Cauchy-Schwarz inequality gives
$$
\nu\|Av\|^2_H\le C_2\|v\|_H^{7/4}\|Av\|^{5/4}_H+\|f\|_H\|Av\|_H.
$$
Then, by the Young inequality, we have
$$
\|Av\|^2_H\le C\|v\|^{14/3}_H+C\|f\|^2_H
$$
and combining this estimate with \eqref{3.l2}, we end up with the desired estimate \eqref{3.disV} and finish the proof of the lemma.
\end{proof}

We now turn to the non-autonomous equation and state the well-posedness and solvability results obtained in \cite{Il_Ti}.
\begin{theorem}\label{Il_Ti}
Let $f\in H$, then problem \eqref{main_mod} is uniquely solvable for any  $\bar{u}_0 \in H^1$. Moreover, for any solution $\bar{u}(t)$  of equation \eqref{main_mod}  the following properties are valid:

1) Dissipativity of $\bar{u}(t)$ in $H^s$, $s=1,2$ :
\begin{equation} \label{dis_bu}
\|\bar{u}(t)\|_{H^s} \le e^{-\gamma t}Q_s(\|\bar{u}(0)\|_{H^s}) + K,
\end{equation}
for some monotone function $Q$ which depends on $s$, $\alpha$ and $\nu$, constant $K$ which depends on $s$,$\alpha$, $\nu$ and $\|f\|_H$, and constant $\gamma$ which depends on $\nu$.

2) Smoothing property:
\begin{equation}\label{smooth}
\|\bar{u}(t)\|_{H^2} \le C t^{-1/2}Q(\|\bar{u}(0)\|_{H^1}) + K,
\end{equation}
for some positive constants $ C, K$ which depend on $\alpha$, $\nu$ and $\|f\|_H$.
\end{theorem}
\begin{proof}
For the sake of completeness we provide a sketch of the proof of the estimates \eqref{dis_bu} and \eqref{smooth}  from \cite{Il_Ti}.
\par
To verify the dissipative estimate for $\bar u(t)\in H^{1}$,
we take the inner product of \eqref{main_mod} with $\bar{u}$ and obtain
\begin{equation}
\frac{1}{2} \frac{d}{dt}(\|\bar{u}\|_H^2 + \alpha^2\|\bar{u}\|_{H^1}^2) + \nu (\|\bar{u}\|^2_{H^1} + \alpha^2 \|\bar{u}\|^2_{H^2}) + (B(u,\bar{u}), \bar{u}) = (f, \bar{u}).
\end{equation}
Since $f \in H$, due to the Cauchy-Schwarz and Young's inequalities, we have
\begin{equation}
|(f,\bar{u})| \le \frac1{2 \nu}\| f\|^2_H + \frac{\nu}{2} \| \bar{u}\|^2_H \le \frac1{2 \nu}\| f\|^2_H + \frac{\nu}{2} \| \bar{u}\|^2_{H^1}.
\end{equation}
Using property \eqref{est_B_3} of the  bilinear form $B$, we get
\begin{equation}\label{bu_in}
\frac{d}{dt}(\|\bar{u}\|^2_H + \alpha^2 \|\bar{u}\|^2_{H^1}) + \nu (\|\bar{u}\|^2_{H^1} + \alpha^2 \|\bar{u}\|^2_{H^2}) \le \frac1{\nu}\| f\|^2_H.
\end{equation}
Using Poincare's inequality, we obtain
\begin{equation}
\frac{d}{dt}(\|\bar{u}\|^2_H + \alpha^2 \|\bar{u}\|^2_{H^1}) + \nu (\|\bar{u}\|^2_{H} + \alpha^2 \|\bar{u}\|^2_{H^1}) \le \frac1{\nu}\| f\|^2_H.
\end{equation}
Finally, application of Gronwall's inequality gives us
\begin{equation}\label{bu}
\|\bar{u}(t)\|^2_H + \alpha^2 \| \bar{u}(t)\|^2_{H^1} \le e^{-\nu t}(\|\bar{u}(0)\|^2_H + \alpha^2 \| \bar{u}(0)\|^2_{H^1}) + C_{\nu} \|f\|^2_H.
\end{equation}
Thus we prove the desired estimate \eqref{dis_bu} for $s=1$. Moreover, integrating \eqref{bu_in} from $t$ to $t+1$ and using \eqref{bu}, we derive that
\begin{equation}\label{bu-av}
\int_t^{t+1}\|\bar u(s)\|^2_{H^1}+\alpha^2\|\bar u(s)\|^2_{H^2}\,ds\le  Ce^{-\nu t}(\|\bar{u}(0)\|^2_H + \alpha^2 \| \bar{u}(0)\|^2_{H^1}) + C \|f\|^2_H.
\end{equation}
Let us deduce now the dissipative estimate for $\bar u\in H^2$.
To this end, we multiply \eqref{main_mod} on $A\bar{u}$. We get
\begin{equation} \label{sm1_H}
\frac{1}{2} \frac{d}{dt}(\|\bar{u}\|^2_{H^1} + \alpha ^2 \|\bar{u}\|^2_{H^2}) + \nu(\|\bar{u}\|^2_{H^2} + \alpha^2\|\bar{u} \|^2_{H^3}) + (B(u,\bar{u}), A\bar{u}) = (f, A \bar{u}).
\end{equation}
Obviously,
\begin{equation}\label{f_H}
| (f, A \bar{u})| \le \frac{1}{2\nu} \|f\|^2_H + \frac{\nu}{2} \|A \bar{u}\|^2_H.
\end{equation}
Due to estimates \eqref{est_B_2}, \eqref{bar_u_u} and H\"older's inequality
\begin{multline}\label{B_H}
|(B(u, \bar{u}),A \bar{u})| \le C \|u\|_H^{1/2} \|u\|^{1/2}_{H^1} \|\bar{u}\|_{H^1} \|A\bar{u}\|_{H^1} \le  \bar{C} \|\bar{u}\|_{H^1} \|\bar{u}\|^{1/2}_{H^2} \| \bar{u}\|^{3/2}_{H^3} \le \\ \frac12C_\nu \|\bar{u}\|^4_{H^1} \|\bar{u}\|^2_{H^2} + \frac{\nu\alpha^2}{2}\|\bar{u} \|^2_{H^3}.
\end{multline}
Substitution of \eqref{f_H} and \eqref{B_H} into \eqref{sm1_H} gives us
\begin{equation}\label{sm2_H}
 \frac{d}{dt}(\|\bar{u}\|^2_{H^1} + \alpha ^2 \|\bar{u}\|^2_{H^2}) + \nu(\|\bar{u}\|^2_{H^2} + \alpha^2\|\bar{u} \|^2_{H^3})  \le \frac{1}{\nu} \|f\|^2_H + C_\nu \|\bar{u}\|^4_{H^1} \|\bar{u}\|^2_{H^2}.
\end{equation}
Applying the Gronwall inequality to \eqref{sm2_H} and using estimates \eqref{bu} and \eqref{bu-av} for estimating the right-hand side of \eqref{sm2_H}, we end up with the desired dissipative estimate for $s=2$.
\par
Let us finally deduce the smoothing property \eqref{smooth}.
To this end, we integrate \eqref{bu_in} from $0$ to $t$ and obtain
\begin{equation}\label{L_2_L_2}
\nu \int_0^t (\|u(\tau)\|^2_{H^1} + \alpha^2 \|\bar{u}(\tau)\|^2_{H^2}) d\tau \le \|\bar{u}(0)\|^2_{H} + \alpha^2 \|\bar{u}(0)\|^2_{H^1} + t \frac{1}{\nu}\|f\|^2_H.
\end{equation}
Multiplying \eqref{sm2_H} by $t$, integrating from $0$ to $t$ and using \eqref{L_2_L_2}, we finally arrive at
\begin{equation}
t (\|\bar{u}(t)\|^2_{H^1} + \alpha ^2 \|\bar{u}(t)\|^2_{H^2})\le  C_{\nu}(\|\bar{u}(0)\|^2_{H} + \alpha^2 \|\bar{u}(0)\|^2_{H^1}) + t Q_1(\|\bar{u}(0)\|_{H^1}) + t Q_2(\|f\|_H),
\end{equation}
and smoothing property \eqref{smooth} is proved and the theorem is also proved.
\end{proof}
\begin{theorem}
Let $f \in H$ then for any solution $u(t)$ of the problem $\eqref{main_mod}$ the following dissipative estimates hold:
\begin{equation}\label{dis_H_s_u}
\|{u}(t)\|_{H^s} \le Q_s (\|{u}(0)\|_{H^s})e^{-\gamma t} + K,
\end{equation}
where $s = 1,2$, the monotone function $Q$ depends on $s, \nu$ and $\alpha$, the positive constant $\gamma$ depends on $\nu$ and $K$ depends on $s, \nu$ and $\|f\|_H $. Moreover the following smoothing property is valid:
\begin{equation}\label{smooth_H_s_u}
\|u(t)\|_{H_s} \le t^{-N_s} Q_s(\|u(0)\|_H) +K,\ \ t>0,
\end{equation}
where as before $s= 1,2$.
\end{theorem}
\begin{proof}
\textbf{$\bf{H^1}$-estimate on $\bf{u}$.}

Let us take the inner product of \eqref{main_mod} with $A u$. Then we obtain
\begin{equation} \label{eq_u_H^1}
\frac{1}{2}\frac{d}{dt} \|u\|^2_{H^1} + \nu \|Au\|^2_H +(B(u,\bar{u}), Au) = (f,Au).
\end{equation}
By the Cauchy-Schwarz and Young's inequalities, we have
\begin{equation}\label{f_est}
|(f,Au)| \le \|f\|_H \|Au\|_H \le c_{\nu}\|f\|^2_H + \frac{\nu}{4}\|Au\|^2_H.
\end{equation}
Then we use estimates \eqref{est_B} and \eqref{bar_u_u} together with interpolation inequality to deduce
\begin{multline}
|(B(u,\bar{u}), Au)| \le c \|u\|_{H^1} \|\bar{u}\|^{1/2}_{H^1} \|A \bar{u}\|^{1/2}_H \|Au\|_H \le \\
c_1\|u\|_{H}^{1/2}\|Au\|_H^{1/2}\|u\|_H^{1/2}\|\bar u\|^{1/2}_{H^{1}}\|Au\|_{H}\le
 c_2 \|u\|^{3/2}_{H} \|Au\|^{3/2}_H \le  c_{\nu}  \|u\|^6_{H} + \frac{\nu}{4}\|Au\|^2_H.
\end{multline}
Consequently, equation \eqref{eq_u_H^1} transforms into
\begin{equation}\label{growing}
\frac{d}{dt}\|u\|^2_{H^1} + \nu \|u\|^2_{H^1} \le 2 c_{\nu}  \|u\|^6_{H} + 2 c_{\nu}\|f\|^2_H,
\end{equation}
Finally, applying  Gronwall's inequality and using estimate \eqref{dis_bu} we get
\begin{equation}\label{dis_qrow}
\|u(t)\|^2_{H^1} \le e^{-\nu t}Q_1( \|u(0)\|_{H^1}) +Q(\|f\|_H),
\end{equation}
which is the desired estimate \eqref{dis_H_s_u}. It remains to prove only the smoothing property \eqref{smooth_H_s_u}. Let us notice that integrating  \eqref{sm2_H} from $0$ to $t$ and using \eqref{dis_bu}, we have
\begin{equation}\label{L_2_H_1}
\int_0^t\|u(\tau)\|^2_{H^1}d\tau \le Q(\|u(0)\|_H) +tQ(\|f\|_H).
\end{equation}
Thus multiplying equation \eqref{growing} by $t$, integrating over the interval $(0,t)$, for $t\le 1$, and using \eqref{L_2_H_1} we come to
\begin{equation}\label{interm}
t \|u(t)\|^2_{H^1}\le \tilde{Q}(\|u(0)\|_H) +t\tilde{Q}(\|f\|_H),
\end{equation}
which is the desired estimate \eqref{smooth_H_s_u} for $s=1$.

\textbf{$\bf{H^2}$-estimate on $\bf{u}$.}

Let us notice that since $f \not \in H^1$  the straightforward multiplication of equation \eqref{main_mod} on $A^2 u$ does not help to deduce the dissipative estimate for the $H^2$-norm. Instead of this we will obtain the dissipative estimate on $\|\pt u(t)\|_H$ which is in fact equivalent to the $\|u(t)\|_{H^2}$. Indeed, estimating the term $\nu Au$ from equation \eqref{main_mod}, we get
\begin{equation}
\nu\|Au(t)\|_H\le\|\Dt u(t)\|_H+\|B(u(t),\bar u(t))\|_H+\|f\|_H,
\end{equation}
so we only need to estimate the $H$-norm of the non-linearity $B(u,\bar u)$ as follows
\begin{equation}\label{3.B_u_bu}
\|B(u(t),\bar u(t))\|_H\le C\|u(t)\|_{L^2}\|\nabla\bar u(t)\|_{L^\infty}\le C_1\|u(t)\|_H\|u(t)\|_{H^1}\le C_1\|u(t)\|_{H^1}^2.
\end{equation}
Therefore,
\begin{equation}\label{pt_u_u_2_1 }
\nu \|Au(t)\|_H \le \|f\|_H + \|\pt u(t)\|_H + C\|u(t)\|_{H^1}^2.
\end{equation}
On the other hand,
\begin{equation}\label{pt_u_u_2_2}
\nu \|A u(t)\|_H \ge\|\pt u(t)\|_H  - \|f + B(u(t), \bar u(t))\|_H \ge\| \pt u(t)\|_H -\|f\|_H-C\|u(t)\|^2_{H^1}
\end{equation}
and keeping in mind already proved estimate \eqref{dis_H_s_u} with $s=1$, we see that the dissipative estimates for $u(t)$ in $H^2$ and for $\|\pt u(t)\|_H$ are equivalent.
\par
To obtain the desired estimate for $\pt u$, we differentiate equation \eqref{main_mod} with respect to $t$ and denote $\phi= \pt u$. This gives
\begin{equation}\label{dif_main_mod}
\pt \phi + \nu A \phi + B(\phi, \bar{u}) + B(u,\bar{\phi}) = 0.
\end{equation}
Taking the inner product with $\phi$, we obtain
\begin{equation} \label{dif_eq_H}
\frac{1}{2}\frac{d}{dt}\|\phi\|^2_H + \nu\|\phi\|^2_{H^1} + (B(\phi, \bar{u}), \phi)+ (B(u,\bar{\phi}),\phi)=0.
\end{equation}
To estimate the last two terms we use property \eqref{est_B} of $B$, Holder's inequality, interpolation inequality
\begin{equation}\label{interpol}
 \|\phi\|^2_H \le \|\phi\|_{H^{-1}} \|\phi\|_{H^1}
\end{equation}
and the Young inequality. Then,
we get
\begin{multline}\label{ap_B_1}
|(B(\phi, \bar{u}), \phi)| \le c \|\phi\|_{H^1}\|\bar{u}\|^{1/2}_{H^1}\|A\bar{u}\|^{1/2}_{H}\|\phi\|_H \le\\\le \|\phi\|^{3/2}_{H^1}\|\bar u\|_{H^1}^{1/2}\|u\|_H^{1/2}\|\phi\|^{1/2}_{H^{-1}}\le \frac{\nu}{4}\|\phi\|^2_{H^1} + C_{\nu}\|\bar{u}\|^2_{H^1} \|u\|^2_H \|\phi\|^2_{H^{-1}}.
\end{multline}
Analogously,
\begin{multline}\label{ap_B_2}
|(B(u,\bar{\phi}),\phi)| \le c \|u\|_{H^1}\|\bar{\phi}\|^{1/2}_{H^1}\|\bar{\phi}\|^{1/2}_{H^2} \|\phi\|_{H}\le\\\le C\|u\|_{H^1}\|\phi\|_{H^{-1}}^{1/2}\|\phi\|_H^{3/2}\le C_1\|u\|_{H^1}\|\phi\|^{5/4}_{H^{-1}}\|\phi\|_{H^1}^{3/4} \le \frac{\nu}{4}\|\phi\|^2_{H^1} + C_{\nu} \|u\|^{8/5}_{H^1}\|\phi\|^2_{H^{-1}}.
\end{multline}
Estimating $\phi=\pt u$ from equation \eqref{main_mod} and using \eqref{3.B_u_bu}, we get
\begin{equation}\label{phi}
\|\phi(t)\|_{H^{-1}} \le \|f\|_{H} + \nu\|u(t)\|_{H^1} + \|B(u(t),\bar{u}(t))\|_{H}\le C\|u(t)\|_{H^1}(1+\|u(t)\|_{H^1})+\|f\|_H.
\end{equation}
Finally, \eqref{ap_B_1} and \eqref{ap_B_2} can be estimated as
\begin{equation} \label{B+B}
|(B(\phi, \bar{u}), \phi)| + |(B(u,\bar{\phi}),\phi)| \le \frac{\nu}{2}\|\phi\|^2_{H^1}+Q(\|u(t)\|_{H^1})+Q(\|f\|_H),
\end{equation}
for some monotone function $Q$.
After substitution of \eqref{B+B}  equation \eqref{dif_eq_H} turns into
\begin{equation} \label{H_2_grow}
\frac{d}{dt}\|\phi\|^2_H + \nu \|\phi\|^2_{H^1} \le Q(\|u(t)\|_{H^1})+Q(\|f\|_H).
\end{equation}
Applying Gronwall's inequality,  estimating the right-hand side by \eqref{dis_H_s_u} with $s=1$ and expressing the value $\pt u(0)$ using  \eqref{pt_u_u_2_2}, we get
\begin{equation}\label{grow_v}
\|\pt u(t)\|_{H} \le  Q_1(\|u(0)\|_{H^2})e^{-\gamma t}+Q_1(\|f\|_H),
\end{equation}
where $\gamma>0$ and $Q_1$ is some monotone function. Together with \eqref{pt_u_u_2_1 } this finishes the proof of the dissipative estimate for the $H^2$-norm of $u(t)$.
To obtain smoothing property \eqref{smooth_H_s_u} we multiply inequality \eqref{H_2_grow} by $t^2$ and integrate over the interval $(0,t)$.
\begin{equation}
t^2 \|\phi\|^2_H + \nu \int_0^t \tau^2 \|\phi\|^2_{H^1}d\tau \le \int_0^t 2 \tau \|\phi\|^2_H d\tau+Q(\|u(0)\|_{H^1})+Q(\|f\|_H).
\end{equation}
Due to interpolation inequality \eqref{interpol} and estimate \eqref{phi}, we have
\begin{multline}
t^2 \|\phi\|^2_H + \nu \int_0^t \tau^2 \|\phi\|^2_{H^1}d\tau \le  \nu \int_0^t  \tau^2 \|\phi\|^2_{H^1} d\tau + C_{\nu}\int_0^t \|\phi\|^2_{H^{-1}} d\tau+\\+Q(\|u(0)\|_{H^1})+Q(\|f\|_H) \le   \nu\int_0^t  \tau^2 \|\phi\|^2_{H^1} d\tau  + Q_1(\|u(0)\|_{H^1}) + Q_1(\|f\|_H).
\end{multline}
Finally, thanks to estimate \eqref{pt_u_u_2_1 } we change $\|\phi(t)\|_H$ on $\|u(t)\|_{H^2}$ and obtain smoothing property \eqref{smooth_H_s_u}. The theorem is proved.
\end{proof}
Let us consider equation on $w:=u - v$, where $u$ solves problem \eqref{main_eq} and $v$ solves stationary problem \eqref{stat_pr} (actually, it may be many solutions of this problem, then we just fix any of them). After applying Helmholtz-Leray projection we get
\begin{equation}\label{main_eq_w}
 \begin{cases}
  \pt w + \nu A w + B(w,\bar{w}) + B(v, \bar{w}) + B(w,\bar{v})= 0,\\
  w =\bar{w} + \alpha^2 A \bar{w},\\
  w(0) = w_0.
 \end{cases}
\end{equation}
In the sequel, we will study the solutions $w$ and the corresponding attractor of the equation \eqref{main_eq_w} instead of solutions $u$ of equation \eqref{main_eq}. The advantage of this approach is that the solutions of \eqref{main_eq_w} are more regular. Indeed, the regularity of the solution $u(t)$ is restricted by the regularity of the external forces $f$ and we cannot expect more regularity than $H^2$ from the solution $u$ if the external force $f\in H$ only. In contrast to this, as we see from the next theorem, the solution $w(t)$ will be at least $H^3$ even if $f\in H$ only.

\begin{theorem}
Let $w(t)$ be a solution of equation \eqref{main_eq_w} with $w(0)\in H$. Then $w(t)\in H^3$ for all $t\ge0$ and the following dissipative estimate holds:
\begin{equation}\label{dis_est_H_3}
\|w(t)\|_{H^3} \le Q(\|w(0)\|_{H^3})e ^{-\gamma t} + \tilde{Q}(\|f\|_H),
\end{equation}
where the monotone function $Q$ and positive constant $\gamma$ depend on $\alpha$ and $\nu$, and monotone function $\tilde{Q}$ depends on $\nu$. Moreover, the following smoothing property is valid:
\begin{equation}\label{smooth_H_3}
\|w(t)\|_{H^3} \le (1+t^{-N})Q(\|w(0)\|_{H}),\ \ t>0,
\end{equation}
where as before monotone function $Q$ depends on $\alpha$ and $\nu$.
\end{theorem}
\begin{proof}
 Let us take the inner product of equation \eqref{main_eq_w} with $A^3 w$. Then, we get
\begin{equation}\label{est_H_3}
\frac{1}{2}\frac{d}{dt} \|w\|^2_{H^3}+\nu \|w\|^2_{H^4}+(B(w,\bar{w}),A^3 w)+(B(w,\bar{v}),A^3 w)+(B(v,\bar{w}), A^3 w)=0.
\end{equation}
Using H\"older's and Young's inequalities and the fact that $H^2(\mathbb T^3)$ is an algebra, we obtain
\begin{multline}\label{B_H_3_a}
|(B(w,\bar{w}),A^3 w)| = |(AB(w,\bar{w}),A^2 w)| \le \frac{1}{\nu}\|B(w,\bar{w})\|^2_{H^2} + \frac{\nu}{4} \|w\|^2_{H^4}\\ \le C\|w\|^2_{H^2}\|w\|^2_{H^1}+ \frac{\nu}{4} \|w\|^2_{H^4},
\end{multline}
where constant $C$ depends on $\nu$ and $\alpha$. Similarly,
\begin{equation}\label{B_H_3_b}
|(B(v,\bar{w}), A^3 w)| + |(B(w,\bar{v}), A^3 w)| \le C \|v\|^2_{H^2} \|w\|^2_{H^2}+ \frac{\nu}{4} \|w\|^2_{H^4},
\end{equation}
where constant $C$ depends on $\nu$ and $\alpha$.
\par
Substituting \eqref{B_H_3_a} and \eqref{B_H_3_b} into \eqref{est_H_3} we obtain
\begin{equation} \label{fin_H_3}
\frac{d}{dt} \|w\|^2_{H^3} + \nu \|w\|^2_{H^4} \le C\|w\|^2_{H^2}(\|w\|^2_{H^1} + \|v\|^2_{H^2}).
\end{equation}
Estimate \eqref{fin_H_3} is similar to \eqref{growing} and \eqref{H_2_grow} so arguing exactly as in the cases of $H^1$- and $H^2$-norms for the equation \eqref{main_mod} we get the desired estimates \eqref{dis_est_H_3} and \eqref{smooth_H_3} for $s=3$. Theorem is proved.
\end{proof}
Our next step is to study the global attractor for the solution semigroup $S(t)$ associated with equation \eqref{main_eq_w}. Let us recall the definition of a global attractor.
\begin{definition}
A set $\mathcal{A}$ is a global attractor for the semigroup $S(t)$ if the following properties are satisfied:

1. $\mathcal{A}$ is compact in $H$;

2. $\mathcal{A}$ is strictly invariant: $S(t)\mathcal{A} = \mathcal{A}$ for all $t \ge 0$;

3. $\mathcal{A}$ is an attracting set for the semigroup $S(t)$, i.e. for every bounded set $B \subset H$ and every neighbourhood $\mathcal{O}(\mathcal{A})$ there exists time $T = T(B)$ such that $S(t)B \subset \mathcal{O}(\mathcal{A})$ for all $t\ge T$.
\end{definition}
Summarizing results of this section we obtain.
\begin{theorem}\label{attractor}
Let $S(t): H \to H$ be the solution semigroup generated by equation \eqref{main_eq_w}. Then this semigroup possesses a global attractor $\mathcal{A}$ in the phase space $H$, which is a bounded set in $H^3$.
\end{theorem}
\begin{proof}
Indeed, from the global attractor's theory (see, e.g., \cites{BV, Tem1}) we know, that it is sufficient to prove  that

1. semigroup $S(t)$ is continuous in H for every fixed $t$;

2. $S(t)$ possesses a compact absorbing set in $H$.

From the dissipative estimate \eqref{dis_est_H_3} and smoothing property \eqref{smooth_H_3} it follows that the ball
\begin{equation*}
\mathcal B:=\{ w \in H, \|w\|^2_{H^3} \le 2 \tilde{Q}(\|f\|_H)\}
\end{equation*}
in $H^3$ is an absorbing set in $H$. Thus, the existence of a compact absorbing set is verified.
To prove the continuity of $S(t)$ in $H$ for every fixed $t$ we consider two solutions $u_1$ and $u_2$  of equation \eqref{main_mod} with initial values $u_1(0)= u_1^0$ and $u_2(0)=u_2^0$, and $v$ a solution of the stationary problem \eqref{stat_pr}. Then subtracting from the equation for $u_1-v$ the equation for $u_2-v$ and denoting $w:= u_1- u_2$ we get
\begin{equation}\label{w_diff}
\pt w + \nu A w + B(w,\bar{u}_1) + B(u_2,\bar{w})=0.
\end{equation}
By taking the inner product of \eqref{w_diff} with $w$,
\begin{equation}
\frac{1}{2}\frac{d}{dt}\|w\|^2_H + \nu \|w\|^2_{H^1} + (B(w,\bar{u}_1),w) + (B(u_2,\bar{w}),w)=0.
\end{equation}
We estimate the non-linear terms using the H\"older inequality with exponents $2$, $3$ and $6$ and the embedding $H^1\subset L^6$. This gives
\begin{multline}
|(B(w,\bar u_1),w)|\le C\|w\|_{L^2}\|\nabla \bar u_1\|_{L^3}\|w\|_{L^6}\le\\\le C_1\|w\|_{H}\|\bar u_1\|_{H^2}\|w\|_{H^1}\le \frac\nu2\|w\|_{H^1}^2+C_\nu\|u_1\|^2_H\|w\|^2_H
\end{multline}
and, analogously,
\begin{equation}
|(B(u_2,\bar w),w)|\le\frac\nu2\|w\|^2_{H^1}+C_\nu\|u_2\|^2_H\|w\|^2_H.
\end{equation}
Thus, inserting these estimates into \eqref{w_diff}, we obtain
\begin{equation}
\frac{d}{dt}\|w\|^2_H \le C_\nu \left(\|u_1\|^2_{H} +\|u_2\|^2_{H}\right) \|w\|^2_H.
\end{equation}
Using Gronwall's inequality and the fact that the $H$-norms of $u_i(t)$ are under control, we finally arrive at
\begin{equation}
\|w(t)\|^2_H \le \|w(0)\|^2_H e^{C \int_0^t \left(\|u_1(s)\|^2_{H}+\|u_2(s)\|^2_H\)\,ds}\le e^{C_1t}\|w(0)\|_{H^2}^2.
\end{equation}
Thus, operator $S(t)$ is continuous in $H$, which together with existence of an absorbing set gives the existence of a global attractor $\mathcal{A}\subset B\subset H^3$.   The theorem is proved.
\end{proof}

\section{Spatial averaging: an abstract scheme}

In this section, we briefly describe the construction of an inertial manifold based on the spatial averaging method in the particular case of zero averaging which will be applied below to the modified Leray-$\alpha$ problem. In this particular case, we suggest an alternative simplified proof of the existence of the inertial manifold which follows the Perron  method and is based on the Banach contraction theorem applied in the properly chosen weighted space of trajectories.  
\par
We consider the following abstract parabolic equation in a Hilbert space $H$:
\begin{equation}\label{ur}
\Dt u+Au=F(u),\ \ u\big|_{t=0}=u_0,
\end{equation}  
where $A$ is a positive definite self-adjoint linear operator with the compact inverse and $F:H\to H$ is the non-linear operator which is assumed to be globally Lipschitz continuous with the Lipschitz constant $L$. Let also $\lambda_n$, $n\in\mathbb N$, be the eigenvalues of the operator $A$ enumerated in the non-decreasing order and $e_n$ be the corresponding eigenvectors. For any fixed $N,k\in\mathbb N$ such that $k<\lambda_N$, we define the orthoprojector to the intermediate modes:
\begin{equation}\label{ukr}
R_{N,k}u:=\sum_{\lambda_N-k<\lambda_n<\lambda_{N+1}+k}(u,e_n)e_n,\ \ u:=\sum_{n=1}^\infty(u,e_n)e_n.
\end{equation}
Finally, we assume that the nonlinearity $F$ is Gateaux differentiable and its derivative $F'(u)\in\mathcal L(H,H)$ satisfies the spatial averaging principle in the following  form: there exists $\rho>0$ such that for every $k\in\mathbb N$ and any $\delta>0$ there are infinitely many numbers $N$ such that $\lambda_{N+1}-\lambda_N\ge\rho$ and 
\begin{equation}\label{av}
\|R_{N,k}\circ F'(u)\circ R_{N,k}v\|_H\le\delta\|v\|_H,\ \ u,v\in H.
\end{equation}
This is  a simplified form of the general averaging principle which corresponds to the case where the spatial averaging of the derivative
 $F$ is identically zero, \cite{mal-par,Zel}. Since in the sequel we need only this particular case for the applications to the Navier-Stokes problem, we will not consider the general case here.
\par
The next theorem can be considered as the main result of the section.

\begin{theorem}\label{mau} Let the above assumptions hold and let the numbers $N$, $k$, $\rho$ and $\delta$ involved into the condition \eqref{av} satisfy the inequality
\begin{equation}\label{unknown}
 \frac{L^2}{k^2}+2\frac{\delta^2}{\rho^2}+2\sqrt{2}\frac{L^2}{k\rho}<1.
\end{equation}
Then, equation \eqref{ur} possesses an $N$-dimensional inertial manifold in the sense of Definition \ref{defman}
\end{theorem}   
\begin{proof} Following the Perron method, we will construct the desired inertial manifold by solving the backward in time problem
\begin{equation}\label{prob-}
\Dt u+Au=F(u),\ \ t\in\mathbb R_-, \ \mathcal P_Nu\big|_{t=0}=u_0^+
\end{equation} 
in the weighted trajectory space $L^2_\theta(\mathbb R_-,H)$ with the norm
$$
\|u\|_{L^2_\theta}^2:=\int_{-\infty}^0e^{2\theta t}\|u(t)\|^2_H\,dt
$$
with $\theta:=\frac{\lambda_N+\lambda_{N+1}}2$. Then the desired Lipschitz continuous function $\Phi: \mathcal P_NH\to \mathcal Q_NH$ is defined by the expression
$$
\Phi(u_0^+):=\mathcal Q_N u(0),
$$ 
where $u(t)$ is a unique solution of \eqref{prob-}, see \cite{Zel} for more details. To solve problem \eqref{prob-} we need to study the associated linear non-homogeneous problem on the whole line $t\in\mathbb R$:
\begin{equation}\label{lin}
\Dt v+Av=h(t),\ \ v,h\in L^2_\theta(\mathbb R,H).
\end{equation}
\begin{lemma}\label{lem} Let $\lambda_{N+1}>\lambda_N$. Then, for any $h\in L^2_\theta(\R,H)$, there is a unique solution $v\in L^2_\theta(\R,H)$ and the following estimate hold:
\begin{equation}\label{useless}
\|v\|_{L^2_\theta}\le\frac{2}{\lambda_{N+1}-\lambda_N}\|h\|_{L^2_\theta}.
\end{equation}
\end{lemma}
\begin {proof}[Proof of the lemma]
The proof of this lemma follows from the estimate for the Fourier modes $v_n$ of the solution $v$:
\begin{equation}\label{lin-n}
\|v_n\|_{L^2_\theta(\R)}^2\le \frac1{(\lambda_n-\theta)^2}\|h_n\|^2_{L^2_\theta(R)}
\end{equation}
which in turn can be easily verified using the explicit formula for the solution $v_n$ of the following ODE
$$
v_n''(t)+\lambda_nv_n(t)=h_n(t).  
$$
Then, using the obvious fact that the maximal value of the factor $\frac{1}{(\lambda_n-\theta)^2}$ is achieved at $n=N$ and at $n=N+1$, we have
\begin{multline}
\|h\|^2_{L^2_\theta(R,H)}=\sum_{n=1}^\infty\|h_n\|^2_{L^2_\theta(\R)}\le\sum_{n=1}^\infty\frac1{(\lambda_n-\theta)^2}\|h_n\|_{L^2_\theta(\R)}^2\le\\\le \frac4{(\lambda_{N+1}-\lambda_N)^2}\sum_{n=1}^\infty\|h_n\|^2_{L^2_\theta(\R)}=\frac4{(\lambda_{N+1}-\lambda_N)^2}\|h\|^2_{L^2_\theta(\R,H)}
\end{multline}
and formula \eqref{useless} is proved.
\end{proof}
At the next step, we transform equation \eqref{prob-} to the equivalent problem defined on the whole line $t\in\R$. Namely, we introduce the new variable $w(t):=u(t)-e^{-At}u_0^+$ and the new non-linearity:
\begin{equation}
\tilde F(w,u_0^+):=\begin{cases}F(w(t)+e^{At}u_0^+),\ t\le0\\ 0,\ \ t>0.\end{cases} 
\end{equation}
Then, as not difficult to see, problem \eqref{prob-} is equivalent to the following problem defined on the whole line $t\in\R$: 
\begin{equation}
\Dt w+Aw=\tilde F(w,u_0^+),\ \ w\in L^2_\theta(\R,H)
\end{equation}
or denoting by $\mathcal T: L^2_\theta(\R,H)\to L^2_\theta(\R,H)$ the solution operator of problem \eqref{lin}, we arrive at
\begin{equation}\label{eqban}
w=\mathcal T\circ\tilde F(w,u_0^+):=G(w,u_0^+), \ \ w\in L^2_\theta(\R,H),
\end{equation}
see \cite{Zel} for the details. We want to solve \eqref{eqban} using the Banach contraction theorem and to this end we only need to check that the right-hand side is globally Lipschitz with Lipschitz constant less than one. Note that the Function $\tilde F$ is obviously globally Lipschitz in $L^2(\R,H)$ with Lipschitz constant $L$ (since $F$ is Lipschitz in $H$ with the same constant), so the  estimate based on Lemma \ref{lem} and estimate \eqref{useless} gives    
that the Lipschitz constant of $G$ with respect to $w$ does not exceed $\frac{2L}{\lambda_{N+1}-\lambda_N}$, so the existence of a solution is straightforward if
\begin{equation}\label{gap}
\lambda_{N+1}-\lambda_N>2L
\end{equation}
and this is the classical spectral gap assumption in the theory of inertial manifolds.
\par 
In our case, the spectral gap assumption  is not assumed to be satisfied and $G$ may be not a contraction in the initial norm in the space $H$. However, the spatial averaging assumption \eqref{av} together with estimates \eqref{lin-n} allows us to verify the contraction in the properly chosen equivalent norm in $H$. Namely, let
\begin{equation}\label{norm}
\|u\|_\eb^2:=\eb^2\|R_{N,k}u\|^2_H+\|(1-R_{N,k})u\|_H^2.
\end{equation}
Then, using \eqref{lin-n}, we see that
$$
\|(1-R_{N,k})\mathcal T\circ \tilde F'_w(u)z\|_{L^2_\theta(\R,H)}^2\le \frac {L^2}{k^2}\|z\|^2_{L^2_\theta(\R,H)}\le \frac{L^2}{k^2}\(1+\frac1{\eb^2}\)\|z\|^2_{L^2_\theta(\R,H_\eb)}.
$$
Using now the spatial averaging assumption, we estimate the intermediate part as follows:
\begin{multline}
\eb^2\|R_{N,k}\mathcal T\circ F'_w(u)z\|^2_{L^2_\theta(\R,H)}\le2\eb^2
\|R_{N,k}\mathcal T\circ F'_w(u)R_{N,k}z\|^2_{L^2_\theta(\R,H)}+\\+
2\eb^2\|R_{N,k}\mathcal T\circ F'_w(u)
(1-R_{N,k})z\|^2_{L^2_\theta(\R,H)}\le2\eb^2\frac{\delta^2}{\rho^2}
\|R_{N,k}z\|^2_{L^2_\theta(\R,H)}+2\eb^2\frac{L^2}{\rho^2}
\|(1-R_{N,k})z\|^2_{L^2_\theta(\R,H)}\le\\\le 
2\frac{\delta^2+\eb^2L^2}{\rho^2}\|z\|^2_{L^2_\theta(\R,H_\eb)}.
\end{multline}
Thus, combining two last estimates, we get
\begin{equation}\label{lip}
\|\mathcal T\circ \tilde F'_w(u)z\|_{L^2_\theta(\R,H_\eb)}^2
\le \(\frac{L^2}{k^2}\(1+\frac1{\eb^2}\)+2\frac{\delta^2+\eb^2L^2}{\rho^2}\)\|z\|^2_{L^2_\theta(\R,H_\eb)}.
\end{equation}
Fixing now $\eb>0$ in an optimal way, we finally get
\begin{equation}
\|G'_w(w,u_0^+)\|_{\mathcal L(L^2_\theta(\R,H_\eb),L^2_\theta(\R,H_\eb))}^2\le \frac{L^2}{k^2}+2\frac{\delta^2}{\rho^2}+2\sqrt{2}\frac{L^2}{k\rho}
\end{equation}
and assumption \eqref{unknown} guarantees that $G(w,u_0^+)$ is a contraction in the space $L^2_\theta(\R,H_\eb)$ and therefore, by Banach contraction theorem there is a unique solution of problem \eqref{prob-}  for every $u_0^+\in P_NH$. This gives the existence of the inertial manifold. Its Lipschitz continuity and exponential tracking are verified repeating word by word the arguments given in \cite{Zel}, so we omit them here. Theorem \ref{mau} is proved. 
\end{proof} 
 
\section{Existence of an inertial manifold: the case of Navier-Stokes equations}

In this section, we apply the abstract Theorem \ref{mau} to the modified  Leray-$\alpha$ model. First, we adopt this theorem to the case of $3D$ torus. Namely, in slight abuse of the previous notations, we define the projector $R_{N,k}$ to the intermediate modes as follows:
 let $N \in \mathbb{N}$ and $k > 0$ be such that $ N > k$, then
\begin{equation}
R_{k,N}u = \sum_{n \in \mathbb{Z}^3 : N-k\le |n|^2 \le N +k+1} u_n e^{i n \cdot x}.
\end{equation}

Then, we  consider the following abstract model
\begin{equation}\label{mod}
\pt w +\nu A w + \mathcal{F}(w)=0,
\end{equation}
where $\mathcal{F}(w)$ is some non-linear function and $A$ as before is the Stokes operator. We recall that the spectrum of the Stokes operator $A$ consists of all integers which can be presented as a sum of 3 squares, so by the Legendre 3 square theorem, of all $n$'s which are not of the form $n=4^a(8b+7)$. Thus, we do not have gaps in the spectrum of length more than 3 and the typical distance between  the subsequent non-equal eigenvalues is one. Applying the abstract Theorem \ref{mau} to this particular case, we get the following theorem.  
\begin{theorem}\label{main_th}
Let non-linearity $\mathcal{F}(w): H \to H$ be globally bounded, globally Lipschitz continuous and Gateaux differentiable and
 for every $\delta>0$ and $k>0$ there exist infinitely many values of $N\in \mathbb{N}$ satisfying
\begin{equation}\label{est_to_be_checked}
\|R_{k,N}\circ \mathcal{F}'(w)\circ R_{k,N}v\|_H \le \delta\|v\|_H,
\end{equation}
for every $w, v \in H$.
Assume also that
\begin{equation}
 \frac{L^2}{k^2}+2\frac{\delta^2}{\nu^2}+2\sqrt{2}\frac{L^2}{k\nu}<1,
\end{equation}
where $L$ is a Lipschitz constant of the non-linearity.
Then, for all such $N$'s there exists a finite dimensional Lipschitz continuous inertial manifold for problem \eqref{mod} which is the graph over the linear subspace generated by all Fourier modes $e^{ix.n}$ satisfying $|n|^2\le N$.
\end{theorem}

\begin{remark} Let us notice that in contrast to the cases of reaction-diffusion  equation and Cahn-Hilliard equation which are considered in  \cite{mal-par} and \cite{Zel-Ko}, respectively, in order to satisfy conditions of Theorem \ref{main_th} we are going to apply the cut-off procedure in the Fourier space instead of the physical space  that were proposed earlier.
 The choice of the truncation function is dictated by the specific form of the non-linearity and guaranties the validity of the estimate \eqref{est_to_be_checked} in the form which is written here without adding the dependence on the higher norms of $w$.
\end{remark}

Obviously, the non-linearity  $F(w):= B(w,\bar{w})+ B(w,\bar{v})+ B(v,\bar{w})$ does not satisfy conditions of Theorem \ref{main_th} and therefore this theorem is not applicable directly.
 According to Theorem \ref{attractor}, the set
\begin{equation}\label{abs_ball}
\mathcal{B}_3 = \{ u \in H^3, \|u\|_{H^3}\le C_*\}
\end{equation}
is an absorbing ball. Since we are mainly interested in the long-time behaviour of a solution of the \eqref{main_eq_w} we may freely modify the non-linearity $F(w)$ outside of $\mathcal{B}_3$.
 To do this, we introduce the following cut-off function $\theta\in C_0^\infty(\mathbb C)$ such that
\begin{equation}
\theta(\xi) = \xi \text{ when } |\xi| \le 1\ \text{ and }\  |\theta(\xi)|\le 2,\
 \forall\xi\in\mathbb C.
\end{equation}
 and let $\vec \theta(\xi):= (\theta(\xi^1), \theta (\xi^2), \theta(\xi^3))^t\in\mathbb C^3$, where $\xi = (\xi^1, \xi^2, \xi^3)\in\mathbb C^3$ be the corresponding vector-valued version of this cut off function.


Then we cut off the unknown variable in the non-linear terms in \eqref{main_eq_w} by replacing it by $W(w)$, where
the function $W:H\to H$ is defined as follows:
\begin{equation}\label{3.wdef}
W(w):= \sum_{j\in \mathbb{Z}^3, j \not = 0} \frac{C_*}{|j|^3}P_j\vec\theta\left(\frac{|j|^3 w_j}{C_*}\right)e^{ix \cdot j},
\end{equation}
where $C_*$ is the size of the absorbing ball \eqref{abs_ball}, $P_j$ are the Leray projector matrices defined by \eqref{0.pmatrix} and
$$
w=\sum_{j\in \mathbb{Z}^3, j \not = 0}w_je^{ix.j}
$$
(we remind that $w=(w^1,w^2,w^3)$ is a vector variable and each coefficient $w_j=(w_j^1,w_j^2,w_j^3)\in\mathbb C^3$ is also a vector). The next lemma shows that this cut off procedure does not change the equation on the attractor $\mathcal A$.
\begin{lemma}\label{Lem3.cut} For every $w\in\mathcal B_3$ (where $\mathcal B_3$ is an absorbing ball of the equation \eqref{main_eq_w} defined by \eqref{abs_ball}), we have
$$
W(w)=w.
$$
\end{lemma}
\begin{proof} Indeed, according to \eqref{abs_ball},
$$
\sum_{j\ne0}|j|^6|w_j|^2\le C_*^2
$$
and therefore
$$
|w_j^k|\le\frac {C_*}{|j|^3},\ \ k=1,2,3,\ \ j\in\mathbb Z^3\backslash\{0\}.
$$
Thus $\theta(|j|^3w_j^k/C_*)=|j|^3w_j^k/C_*$ and $W(w)=w$.
\end{proof}
The basic properties of the map $W$ are collected in the following two lemmas.
\begin{lemma}\label{Lem3.b-lip} The map $W$ is globally bounded as a map from $H$ to $H^{3/2-\eb}$ for all $\eb>0$, i.e.,
\begin{equation}\label{bound_F}
\|W(w)\|_{H^{3/2-\eb}}\le C_\eb,
\end{equation}
where the constant $C_\eb$ depends on $\eb>0$, but is independent of $w\in H$. Moreover, the map $W$ is globally Lipschitz as a map from $H$ to $H$:
\begin{equation}\label{Lipsch_F}
\|W(w_1)-W(w_2)\|_H\le L_1\|w_1-w_2\|_H,
\end{equation}
where the constant $L_1$ is independent of $w_i\in H$.
\end{lemma}
\begin{proof} Indeed, from the explicit form of the $W(w)$ and the fact that $\vec\theta$ is bounded, we have
\begin{equation}\label{bound_F1}
\|W(w)\|^2_{H^{3/2 - \varepsilon}}  \le 4\sum_{j\in \mathbb{Z}^3, j \not = 0}|j|^{3-2\eb}\frac{C_*^2} {|j|^6}< C_1 \sum_{j\in \mathbb{Z}^3, j \not = 0}\frac{1}{|j|^{3+2\varepsilon}}\le C_\eb
\end{equation}
and the global boundedness is proved. Analogously, using that $\vec{\theta}'$ is globally bounded, we have
$$
\bigg|\vec\theta\left(\frac{|j|^3 (w_1)_j}{C_*}\right) - \vec\theta\left(\frac{|j|^3 (w_2)_j}{C_*}\right)\bigg|\le
K\frac{|j|^3}{C_*}|(w_1)_j-(w_2)_j|
$$
and, consequently,
\begin{multline}\label{Lipsch_F1}
\|W(w_1)- W(w_2)\|_{H}^2\le
\sum_{j\in \mathbb{Z}^3, j \not = 0}\frac{C_*^2}{|j|^6}\bigg|\vec\theta\left(\frac{|j|^3 (w_1)_j}{C_*}\right) - \vec\theta\left(\frac{|j|^3 (w_2)_j}{C_*}\right)\bigg|^2 \le\\\le L_1^2\sum_{j\in \mathbb{Z}^3, j \not = 0}|(w_1)_j-(w_2)_j|^2=L_1^2 \|w_1 - w_2\|^2_H.
\end{multline}
\end{proof}
\begin{lemma}\label{Lem3.dif} The map $W$ is Gateaux differentiable as the map from $H$ to $H$, its derivative is given by the expression
\begin{equation}\label{3.wpdef}
W'(w)z:= \sum_{j\in \mathbb{Z}^3, j \not = 0} P_j\vec\theta'\left(\frac{|j|^3 w_j}{C_*}\right)z_je^{ix \cdot j},
\end{equation}
and $W'(w)\in\mathcal L(H,H)$ is globally bounded in the following sense:
\begin{equation}\label{est-der}
\|W'(w)\|_{\mathcal L(H,H)}\le L_1,
\end{equation}
where $L_1$ is independent of $w\in H$. Moreover, for every $\delta>0$ and every $w_i\in H^{3}$,
\begin{equation}\label{3.hol}
\|W(w_1)-W(w_2)-W'(w_1)(w_1-w_2)\|_H\le L_2\|w_1-w_2\|_H\|w_1-w_2\|_{H^{3}}^\delta.
\end{equation}
\end{lemma}
\begin{proof} The fact that the linear operator $W'(w)$ is bounded and estimate \eqref{est-der} are immediate corollaries of the fact that $\vec\theta'$ is globally bounded. Let us check that $W'(w)$ is Gateaux derivative of $W$. To this end, we note that
\begin{multline}\label{3.dlin}
\frac{C_*}{|j|^3}\(\vec\theta(|j|^3(w_j+\eb z_j)/C_*)-\vec\theta(|j|^3w_j/C_*)\)-\eb\vec\theta'(|j|^3w_j/C_*)z_j=\\=
\eb z_j\int_0^1[\vec\theta'(|j|^3(w_j+\eb s z_j)/C_*)-\vec\theta'(|j|^3w_j/C_*)]\,ds
\end{multline}
and, therefore,
\begin{multline}\label{3.dlin2}
\bigg\|\frac{W(w+\eb z)-W(w)}\eb-W'(w)z\bigg\|_H^2\le\\\le \sum_{j\in\mathbb Z^3\,j\ne0}|z_j|^2\int_0^1\big|\vec\theta'(|j|^3(w_j+\eb s z_j)/C_*)-\vec\theta'(|j|^3w_j/C_*)\big|^2\,ds.
\end{multline}
To verify the differentiability, we need to check that, for every fixed $w,z\in H$ the right-hand side of \eqref{3.dlin2} tends to zero as $\eb\to0$. To this end, we fix a large $N$, split the sum in two pieces and use that $\vec\theta'$ is globally bounded:
\begin{multline}
\sum_{j\in\mathbb Z^3\,j\ne0}|z_j|^2\int_0^1\big|\vec\theta'(|j|^3(w_j+\eb s z_j)/C_*)-\vec\theta'(|j|^3w_j/C_*)\big|^2\,ds=\\=\sum_{j\in\mathbb Z^3\,j\ne0,|j|\le N}+\sum_{j\in\mathbb Z^3\,|j|> N}\le\\\le \sum_{j\in\mathbb Z^3\,j\ne0\,|j|\le N}|z_j|^2\int_0^1\big|\vec\theta'(|j|^3(w_j+\eb s z_j)/C_*)-\vec\theta'(|j|^3w_j/C_*)\big|^2\,ds+
K\sum_{j\in\mathbb Z^3\,,|j|> N}|z_j|^2.
\end{multline}
Since $\vec\theta'$ is continuous, the first term in the right-hand side tends to zero as $\eb\to0$ for every fixed $N$ and the second term can be made arbitrarily small by the choice of $N$ since $z\in H$ is fixed. This proves the convergence to zero of the left-hand side of \eqref{3.dlin2} and the Gateaux differentiability is proved.
\par
Let us prove the H\"older continuity \eqref{3.hol}. Denoting $w=w_2$, $z=w_1-w_2$ and using that $\vec\theta''$ is globally bounded, we have
\begin{multline}\label{3.dlin4}
\bigg|\frac{C_*}{|j|^3}\(\vec\theta(|j|^3(w_j+ z_j)/C_*)-\vec\theta(|j|^3w_j/C_*)\)-\vec\theta'(|j|^3w_j/C_*)z_j\bigg|=\\=
\bigg| z_j\int_0^1[\vec\theta'(|j|^3(w_j+ s z_j)/C_*)-\vec\theta'(|j|^3w_j/C_*)]\,ds\bigg|\le C|j|^3|z_j|^2.
\end{multline}
Therefore,
\begin{multline}
\|W(w+z)-W(z)-W'(w)z\|_H^2\le\\\le K\sum_{j\in\mathbb Z^3\, j\ne0}|j|^6|z_j|^4\le \max_j\{|z_j|^2\}\sum_{j\in\mathbb Z^3\, j\ne0}|j|^6|z_j|^2\le K\|z\|^2_{H}\|z\|^2_{H^3}.
\end{multline}
Thus, inequality \eqref{3.hol} is verified for $\delta=1$. For $\delta=0$ it follows from \eqref{Lipsch_F} and \eqref{est-der}. For intermediate values $0<\delta<1$ it can be obtained by interpolation and the lemma is proved.
\end{proof}

%

Finally, we introduce the following cut-off version of equation \eqref{main_eq_w}
\begin{equation}\label{main_eq_w_mod}
\pt w + \nu A w + F(W(w)) = 0,
\end{equation}
where
$$
F(W) = B(W,\bar{W}) +B(W,\bar{v})+ B(v,\bar{W}),\ \ \bar W:=(1-\alpha\Delta)^{-1}W.
$$
Then, due to Lemma \ref{Lem3.cut}, equations \eqref{main_eq_w} and \eqref{main_eq_w_mod} coincide on the absorbing ball $\mathcal B_3$ and, in particular, on the attractor $\mathcal A$ of the initial problem \eqref{main_eq_w}, so it is sufficient to prove the existence of an inertial manifold for  the  truncated equation \eqref{main_eq_w}.
\par
Let us verify the assumptions of Theorem \ref{main_th} for equation \eqref{main_eq_w}. We start with the analytic properties of the map $F$.

\begin{proposition}\label{Prop3.Flip}
The modified non-linearity $F(W)$ is globally bounded, globally Lipschitz continuous and Gateaux differentiable as a map from $H$ to $H$.
\end{proposition}
\begin{proof} These assertions follow in a straightforward way from the analogous properties of the map $W$ proved above and the standard estimates for the quadratic form $B$. Indeed,

\emph{Global boundedness.}
Due to Sobolev's embedding $H^2 \subset L^\infty$, we have
\begin{equation}\label{3.B1}
\|B(p,q)\|_H\le C\|p\|_{H}\|\nabla q\|_{L^\infty}\le  C_1\|p\|_H\|q\|_{H^3},\ \ p\in H,\ \ q\in H^3
\end{equation}
which together with estimate \eqref{bound_F} and maximal $(H^1,H^3)$-regularity for operator $(1-\alpha\Delta)$ gives
\begin{multline}
\|B(W,\bar{W})  + B(v, \bar{W}) + B(W,\bar{v}) \|_H \le\\\le
C(\|W\|_H \|\bar{W}\|_{H^3} + \|v\|_H \|\bar{W}\|_{H^3} + \|W\|_H \| \bar{v}\|_{H^3})\le  {C}_{\alpha}(\|v\|_{H^1} +1), \ \ w \in H,
\end{multline}
where constant $C_\alpha$ depends on $\alpha$ but independent of $\|w\|_H$.

\emph{Global Lipschitz continuity.} Let us take $w_1, w_2 \in H$ and denote $W_i:=W(w_i)$, then
\begin{multline}\label{Lip}
\|B(W_1,\bar W_1)+B(v,\bar{W}_1)+B(W_1,\bar v)-B(W_2,\bar W_2)-B(v,\bar W_2)-B(W_2,\bar v)\|_H = \\
\| B(W_1-W_2,\bar W_1)+B(W_2,\bar W_1-\bar W_2)+B(v,\bar W_1-\bar W_2)+B(W_1-W_2,\bar v)\|_H.
\end{multline}
As before in order to estimate the first and the fourth terms in the right-hand side of \eqref{Lip}, we use estimate \eqref{3.B1} for the quadratic form $B$ and estimates \eqref{bound_F} and \eqref{Lipsch_F}:
\begin{multline}\label{Lip_1}
\|B(W_1- W_2,\bar W_1)\|_H+\|B(W_1-W_2,\bar v)\|_H \le C
\|W_1 -W_2\|_H\(\|\bar W_1\|_{H^3}+\|\bar v\|_{H^3}\)\le\\\ \le C_1 \|w_1 - w_2\|_H\(\|W_1\|_{H^1}+\|v\|_{H^1}\)\le \frac12L\|w_1-w_2\|_{H},
\end{multline}
where the constant $L$ depends on $\alpha$ and $\|v\|_{H^1}$.
\par
To estimate the remaining terms, we use the following estimate for the quadratic form $B$:
\begin{equation}\label{3.B2}
\|B(p,q)\|_{H}\le C\|p\|_{L^3}\|\nabla q\|_{L^6}\le C_1\|p\|_{H^1}\|q\|_{H^2},
\end{equation}
where we have used H\"older's inequality with exponents $3$ and $3/2$ and Sobolev's embedding $H^1 \subset L^p$, for $1\le p \le 6$. Then, due to estimates \eqref{bound_F} and \eqref{Lipsch_F}, we have
\begin{multline}\label{Lip_2}
\|B(W_2,\bar W_1-\bar W_2)\|_H+\|B(v,\bar W_1-\bar W_2)\|_H\le\\\le C\(\|W_2\|_{H^1}+\|v\|_{H^1}\) \|\bar W_1- \bar W_2\|_{H^2}\le \frac12L \|w_1 - w_2\|_H.
\end{multline}
Combining estimates \eqref{Lip_1} and \eqref{Lip_2},  we obtain the desired Lipschitz continuity of the modified non-linearity $F$:
\begin{equation}
\|F(W(w_1)) - F(W(w_2))\|_H \le L\|w_1 - w_2\|_H.
\end{equation}
Since the Gateaux differentiability of $F$ is an immediate corollary of the differentiability of $W$, the proposition is proved.
\end{proof}
Hence to prove the existence of an inertial manifold for the modified-Leray-$\alpha$ model it remains to show the validity of the inequality \eqref{est_to_be_checked}  only.
According to Lemma \ref{Lem3.dif} and the chain rule, the Gateaux derivative of the non-linearity $F(W)$ can be written as
\begin{equation} \label{Frechet}
F'(w)z = P\(\left((W'(w)z)\cdot \nabla \right)(\bar W(w)+\bar{v})+ \left((W(w)+v)\cdot \nabla\right)(\bar W'(w)z)\)
\end{equation}
and, analogously to the proof of Proposition \ref{Prop3.Flip}, we see that $F'(w)$ is a bounded linear operator in $H$ satisfying
$$
\|F'(w)\|_{\mathcal L(H,H)}\le L
$$
for all $w\in H$.
For simplicity, we denote by
\begin{equation}\label{h}
h=h(w)z: = W'(w)z
\end{equation}
then \eqref{Frechet} may be rewritten as
\begin{equation}
F'(W)z=B(h,\bar W+\bar v)+B(W+v,\bar h)=P\(\left(h\cdot \nabla \right)(\bar{W}+\bar{v}) +\left(( W + v )\cdot \nabla\right)\bar h\),
\end{equation}
where, as usual, $\bar h:=(1-\alpha\Delta)^{-1}h$.
\par
Analogously to \cite{mal-par}, we use the following result from harmonic analysis in order to verify  the key estimate \eqref{est_to_be_checked}.
\begin{theorem}
Let
\begin{equation}
\mathcal{C}_N^k:=\{l \in \mathbb{Z}^3 : N-k\le |l|^2\le N +k\}, \ B_\rho:=\{l \in \mathbb{Z}^3: |l|^2 \le \rho\}.
\end{equation}
Then for any $k>0$ and $\rho>0$ there are infinitely many $N \in \mathbb{N}$ such that
\begin{equation} \label{num_th}
\left(\mathcal{C}_N^k - \mathcal{C}_N^k\right)\cap B_\rho = \{0\}.
\end{equation}
\end{theorem}
\begin{proof}
The proof of the theorem is given in \cite{mal-par}.
\end{proof}
Since the Leray projector $P$ is diagonal in the Fourier basis and up to this projector, the operator $F'(W)z$ is a sum of multiplication operators, it can be written in the Fourier basis as a convolution type operator. Then, as not difficult to show using also that the mean values of $W(w)$ and $v$ are zeros, the property \eqref{num_th} implies that
\begin{multline}\label{numb}
R_{k,N}(F'(w)R_{k,N}z)= R_{k,N}B(R_{k,N}h,\bar W+\bar v)+R_{k,N}B(W+v,R_{k,N}\bar h)=\\=R_{k,N}B(R_{k,N}h,\bar W_{>\rho}+\bar v_{>\rho})+
B(W_{>\rho}+v_{>\rho},R_{k,N}\bar h),
\end{multline}
where  $u_{>\rho}:=\sum_{|l|^2 >\rho}u_l e^{i l \cdot x}$.
Therefore, due to estimates \eqref{3.B1} and \eqref{3.B2} and the fact that $W'(w)$ is a bounded operator in $H$, we get
\begin{multline} \label{fin}
\|R_{k,N}(F'(w)R_{k,N}h)\|_H\le C\|\bar W_{>\rho}+\bar v_{>\rho}\|_{H^3}\|R_{k,N}h\|_H+\\+C\|W_{>\rho}+v_{>\rho}\|_{H^1}\|R_{k,N}\bar h\|_{H^2}\le
C_1\|W_{>\rho}+v_{>\rho}\|_{H^1}\|z\|_H.
\end{multline}
Moreover, due to the interpolation and estimate \eqref{bound_F},
\begin{equation}\label{W_2}
\|W_{>\rho}\|_{H^1}\le C \|W_{>\rho}\|_H^{1-\beta}\|W_{>\rho}\|_{H^{\frac{3}{2}-\kappa}}^\beta\le C \rho^{\frac{-(1-\beta)(\frac{3}{2}-\kappa)}{2}}\|{W}\|_{H^{\frac{3}{2} - \kappa}} \le \tilde{C} \rho^{\frac{-(1-\beta)(\frac{3}{2}-\kappa)}{2}},
\end{equation}
where $\beta = \frac{2}{3-2\kappa}$, $0<\kappa<\frac12$ and the analogous estimate holds for $v_{>\rho}$ as well. Since $v$ is bounded in $H^2$, we finally arrive at
\begin{equation}
\|R_{k,N}(F'(w)R_{k,N}z)\|_H\le \tilde{C}\rho^{\frac{-1+2\kappa}{4}}\left(1+ \|v\|_{H^2}\right)\|z\|_H\le C\rho^{\frac{-1+2\kappa}4}\|z\|_H.
\end{equation}
Since $\rho$ may be chosen arbitrarily large  the desired estimate \eqref{est_to_be_checked} is obtained with $a(w) = 0$. Therefore we proved the following theorem which is the main result of the paper.
\begin{theorem}
For infinitely many values of $N \in \mathbb{N}$ there exists an $N$-dimensional Lipschitz continuous inertial manifold for the modified-Leray-$\alpha$ model \eqref{main_eq} on a $3D$ torus $\mathbb{T} = [-\pi, \pi]^3$ which contains a global attractor.
\end{theorem}
\begin{remark} It is not difficult to show based on estimate \eqref{3.hol} that the non-linearity $F$ satisfies the estimate
$$
\|F(w_1)-F(w_2)-F'(w_1)(w_1-w_2)\|_H\le C\|w_1-w_2\|_H\|w_1-w_2\|_{H^3}
$$
which in turn allows us to verify that the inertial manifolds $\mathcal M=\mathcal M_N$ are not only Lipschitz continuous, but also $C^{1+\eb}$-smooth for some small $\eb=\eb(N)>0$, see \cite{Zel-Ko} for more details.
\end{remark}

\begin{remark}
Let us consider the modified Navier-Stokes-$\alpha$ (NS-$\alpha$) model
\begin{equation}\label{LANS_a}
\begin{cases}
\partial_t u + \nu \Delta u - u \times (\nabla \times \bar{u}) + \nabla p = g(x),\\
\nabla \cdot u = \nabla \cdot \bar{u}  = 0,\\
  u= \bar{u}- \alpha^2 \Delta \bar{u},\\
  u(x,0)=u_0(x).
\end{cases}
\end{equation}
Writing the non-linearity in a more explicit way, we obtain
\begin{equation}
-u \times (\nabla \times \bar{u})=(u\cdot\nabla)\bar{u} - \sum_{j=1}^3 u^j \nabla \bar{u}^j.
\end{equation}
Thus, we see that the first term of the non-linearity coincides with the non-linearity of the modified Leray-$\alpha$ model and the second one also has a similar structure. Therefore applying the same cut-off function and acting exactly  as in the case of the modified Leray-$\alpha$ model we may prove the existence of an inertial manifold for the modified NS-$\alpha$ model.
\begin{theorem}
There exist infinitely many values of $N$'s such that inequality \eqref{est_to_be_checked} is fulfilled for the modified NS-$\alpha$ model \eqref{LANS_a} in a three-dimensional case with periodic boundary conditions. Therefore it possesses a finite-dimensional inertial manifold.
\end{theorem}
\end{remark}
\begin{remark}
Since the non-linearity in the simplified Bardina model
\begin{equation}\label{bar}
 \begin{cases}
  u_t - \nu \Delta u + (\bar{u} \cdot \nabla)\bar{u} + \nabla p=f,\\
 \nabla \cdot u = \nabla \cdot \bar{u}  = 0,\\
  u= \bar{u}- \alpha^2 \Delta \bar{u},\\
  u(0)=u_0
 \end{cases}
\end{equation}
is even milder than in the case of the modified Leray-$\alpha$ model it is obvious that the arguments provided in this paper will lead to the following theorem.
\begin{theorem}
The simplified Bardina model \eqref{bar} on a 3D torus possesses an $N$-dimensional inertial manifold for infinitely many values of $N$'s.
\end{theorem}
\end{remark}
\begin{remark} To conclude, we mention that the above described method works also for the following model with milder filtering of higher modes
\begin{equation}\label{12-model}
 \begin{cases}
  u_t - \nu \Delta u + (u \cdot \nabla)\bar{u} + \nabla p=f,\\
  u= \bar{u}+ \alpha^2 (-\Delta)^{\frac{1}{2}} \bar{u},\\
  \nabla \cdot u = \nabla \cdot \bar{u}  = 0,\\
  u(0)=u_0.
 \end{cases}
\end{equation}
Indeed, on the one hand, as not difficult to show that the exponent $1/2$ in the second equation of \eqref{12-model} is enough to verify the global well-posedness and regularity of solutions (actually, $1/2$ is {\it exactly} the  critical exponent here and using $(-\Delta)^{\theta}$ with $\theta<\frac12$ will lead to the uniqueness problems similar to the case of classical NSE in 3D). On the other hand, the non-linearity here is still "zero order" under this choice of the exponent $\theta=1/2$ and the spatial averaging method works (actually, $\theta=1/2$ looks again as a critical exponent for the applicability of the spatial averaging method even in the 2D case). We will return to this problem somewhere else. 
\end{remark}

\section*{Acknowledgements}
The author is thankful to Prof. Sergey Zelik, Prof. Vladimir Chepyzhov and Prof. Edriss S. Titi for fruitful discussions and sharing their insights and ideas.

\bibliography{References}

\end{document}